\documentclass[12pt]{amsart}
\usepackage{amsmath,amssymb,amsthm,amsfonts,amsxtra}

\newtheorem{theorem}{Theorem} 
\newtheorem{proposition}{Proposition}
\newtheorem{lemma}{Lemma}
\theoremstyle{remark}
\newtheorem{remark}{Remark}

\usepackage[all]{xy}
\newcommand{\set}[2]{\left\{ #1 \,;\, #2\right\}}
\newcommand{\Z}{\mathbb{Z}}
\newcommand{\R}{\mathbb{R}}
\newcommand{\C}{\mathbb{C}}

\begin{document}

\title[Cohomological Laplace transform]{%
Cohomological Laplace transform on non-convex cones and Hardy spaces of $\bar{\partial}$-cohomology on
non-convex tube domains} 
\thanks{The second author was supported by JSPS KAKENHI Grant Number 16K05174.}

\dedicatory{Memory of Gennadi Henkin}

\author[S. Gindikin]{Simon Gindikin}
\address[S. Gindikin]{%
Department of Mathematics, Hill Center, Rutgers University, 110 Frelinghysen Road, Piscataway,
NJ 08854-8019, USA}
\email{gindikin@math.rutgers.edu} 

\author[H. Ishi]{Hideyuki Ishi}
\address[H. Ishi]{%
Graduate School of Mathematics, Nagoya University, Furo-cho, Chikusa-ku, Nagoya 464-8602, Japan}
\email{hideyuki@math.nagoya-u.ac.jp}

\begin{abstract}
We consider a class of non-convex cones $V$ in $\R^n$
 which can be presented as (not unique) union of convex cones of some codimension $q$
 which we call the index of non-convexity. 
This class contains non-convex symmetric homogeneous cones studied in \cite{DG93} and \cite{FG96}.
For these cones we consider a construction of dual non-convex cones $V^*$
 and corresponding non-convex tubes $T$
 and define a cohomological Laplace transform from functions at $V$ to $q$-dimensional cohomology of $T$
 using the language of smoothly parameterized \u{C}ech cohomology. 
We give a construction of Hardy space of $q$-dimensional cohomolgy at $T$.
\end{abstract}

\maketitle

\section{Introduction}
The classical Laplace transform can be considered
 as a specification of Fourier transform of functions (distributions)
 with supports on the positive half-line $\R_{>0}$
 as boundary values of holomorphic functions {on} the half-plane {in $\C$}. 
There are many known Paley-Wiener type theorems
 describing spaces of holomorphic functions on the half plane
 as dual spaces of functions (distributions) on the half-line.
Bochner gave multi-dimensional analogues of these constructions
 for convex pointed cones
 (without lines inside),
 see \cite{BM48}, \cite{G92b} and \cite{GV92}.
More detailed constructions of function spaces are possible for convex homogeneous cones 
 (relative to the group of linear automorphisms),
 see \cite{G75} and \cite{RV73}.

This paper is dedicated to some generalization of these results to the case of non-convex cones.
We want to realize in non-convex case the similar logical sequence as in convex case
 but at more complicated environment.
For some non-convex cones, we define dual non-convex cones and associated non-convex tubes. 
For functions (distributions) at non-convex cones we define Laplace transforms with values in $\bar{\partial}$-cohomology at (non-convex) tubes. 
The principal requirement: boundary values of this cohomology must coincide with Fourier transform of initial functions (distributions). 

Of course,
 we do not consider arbitrary non-convex cones
 but suppose some regularity of their geometry.
We want that, in a natural {sense}, these cones can be {built}
 from convex cones of a fixed codimension $q$  which we call \textit{the index of non-convexity}.
Let us have a cone $V$ in $\R^n$ with a connected boundary.
 \textit{The slices} of $V$ are the pointed convex cones of minimal codimension $q$
 that are connected components of intersections of $V$ with subspaces.
We consider slices with orientations.
One subspace can produce several slices.
Apparently the parameter space of slices can be considered as a set which is branching over Grassmannian
 of $(n-q)$-subspaces. 

If we have a cycle of slices,
 then the oriented union of these slices is proportional to $V$
 with an \textit{integer} coefficient
 which can be zero. 
A cycle is called \textit{generating}
 if its dimension is $q$
 and the coefficient of the union is non-zero.
Our basic restriction on $V$ is the existence of a generating cycle.

In good cases,
 slices at the cycle can be taken without intersections
 but they do not need to be such ones
 since, as a result of the orientation,
 there is a compensation in the union of the slices.
The union of slices from a cycle is proportional to $V$ with a coefficient which can be zero.
Non-convex cones can be homogeneous and symmetric (pseudo-Riemannian ones).
This paper prepares a work with such cones and associated pseudo-Hermitian symmetric spaces.
Let us mention two examples:
 the complement to the light cone (more generally, cones restricted by quadrics),
 and the cone of square {matrices} with the positive determinant. 

For functions on $V$,
 as a replacement of the Laplace transform for non-convex case,
 we suggest to consider the Laplace transforms of restrictions of the function to slices.
For each slice we consider dual cone. 
Since the slices have the codimension $q$,
 the dual cones, which are called \textit{wedges}, will be direct products of a pointed cone and $\R^{{q}}$.
We consider their union $V^*$ as the dual cone to $V$,
 and wedges as its maximal convex parts.
So slices and wedges are dual convex objects which we use for the work with non-convex cones.
It is remarkable that the homogeneous cones mentioned above are self-dual $V = V^*$
 (as well as other symmetric cones)
 for an appropriate dot-product.
Let $T$ be the tube domain over the base $V^*$.
This tube $T$ is $q$-linearly concave (Gindikin-Henkin {\cite{GH78}}).
The set of tubes for wedges is a covering of $T$ by Stein manifolds.

Technically we do not work with the set of all slices (which are usually difficult to describe)
 but with any sufficient set of slices
 that contains a generating cycle of slices. 
For many cones,
 including homogeneous ones,
 it is easy to see that
 $V^*$ is independent of a choice of such subsets of slices.
Probably,
 it is true for a broad class of non-convex cones,
 but we do not consider this natural and interesting problem,
 as well as some other natural geometric questions about non-convex cones 
 such as, for example, whether $(V^*)^* = V$.
Our aim is to develop a minimal geometrical background
 sufficient for the construction of the cohomological Laplace transform
 in non-convex cones 
 satisfying the connection stated above with the Fourier transform. 

Self-dual cones are pseudo-Riemannian symmetric spaces.
Classical symmetric cones (with classical groups of automorphisms) were classified by \cite{DG93}.
We discussed two of them in this paper.
Arbitrary non-convex symmetric cones and their connections with Jordan algebras were investigated in \cite{FG96}.

The object which we build
 --- a family of holomorphic functions in different tubes for wedges---
 may look unforeseeable. 
It turns out that
 strong connections between these functions can be expressed in a standard language.
So we have a non-convex tube which is covered by convex tubes (so Stein manifolds)
 and holomorphic functions on them.
These functions satisfy a remarkable system of differential equations
 which appeared at the integral geometry and goes back to F. John
 (see \cite{G92}).
This system of functions can be interpreted as \u{C}ech cocycle for this covering.
Usual \u{C}ech conditions are oriented on coverings with discrete parameters.
However, in our case
 the parameter set is a smooth manifold
 and it is natural to replace the equations at differences of \u{C}ech by differential equations.
Moreover,
 we can transform this object to a closed $q$-form on the parameter space.
For each value of the parameter,
 the $q$-form depends holomorphically on points in the corresponding tube of the covering.
It is the reason why we call this construction the cohomological Laplace transform.
Our cocycles have a special form since functions on them are constants along maximal planes at tubes of the covering.
Such {cocycles} cannot be coboundaries and as a result we have the operator not at a cocycle but at the cohomology.  
This construction of cohomological Laplace transform
 as well as the cohomological language in connection with smoothly parameterized \u{C}ech coverings 
 goes back to \cite{G92}, see also \cite{BEG05}.

Let us remark that $\R^n$ is the joint boundary of all tubes of the covering.
If we consider functions from such a {function} space that we can define boundary values of the Laplace transform at covering tubes{,}
 then we have a cohomology on $\R^n$ and the integral over a generating cycle gives the Fourier transform of original functions.
So the Fourier transform can be interpreted as boundary values of the cohomological Laplace transform 
 ({in} a spirit of hyperfunctions).
This construction has strong conceptual connections with integral geometry, specifically with the Radon-John transform.
{In} particular, the cohomological Laplace transform is very {close} to the Radon-John transform of the Fourier transform.
Finally, we show how to transfer some Paley-Wiener theorems from convex to non-convex case.

{This paper is dedicated to the memory of the outstanding mathematician Gennadi Henkin.
He was a close friend and collaborator of one of authors (S.G.).
They discussed many times problems which are considered at this paper.}  
%
%
%
%
\section{Preliminary: Convex case}
\subsection{Convex cones and tubes}
Let $v$ be a pointed {open} convex cone
 in the vector space $\R^{\ell}_{\tau}$ of $\ell$-dimensional row vectors $\tau$.
Let $v^*$ be the dual cone of $v$ realized in the space $\R^{\ell}_{r}$ of $\ell$-dimensional column vectors $r$,
 where the coupling of $\tau \in \R^{\ell}_{\tau}$ and $r \in \R^{\ell}_r$ is given by the matrix product
 $\tau r = \tau_1 r_1 + \dots + \tau_{\ell} r_{\ell}$.
Namely,
 we define
$$
 v^* := \set{r \in \R^{\ell}_r}{\tau r>0 \quad (\forall \tau \in \overline{v} \setminus \{0\})}. 
$$
Let $t = t(v^*)$ be the tube domain
 $\R^{\ell}_{r} + i v^* \subset \C^{\ell}_p$.
Elements of $t$ are $\ell$-dimensional column complex vectors $p = r + i s$ with $r \in \R^{\ell}_r$ and $s \in v^*$.

\subsection{Function spaces and Paley-Wiener theorems}
Let $\mathcal{S}(\R^{\ell}_{\tau})$ denote the space of Schwartz functions on $\R^{\ell}_{\tau}$,
 and $\mathcal{S}'(\R^{\ell}_{\tau})$ the dual space of $\mathcal{S}(\R^{\ell}_{\tau})$, that is, 
 the space of tempered distributions on $\R^{\ell}_{\tau}$.
Let $\mathcal{S}(v)$ (resp. $L^2(v)$, and $\mathcal{S}'(v)$)
 be the subspace of 
 $\mathcal{S}(\R^{\ell}_{\tau})$ (resp. $L^2(\R^{\ell}_{\tau})$ and $\mathcal{S}'(\R^{\ell}_{\tau})$)
 consisting of functions (or distributions)
 whose support is contained in the closure $\overline{v}$ of the cone $v \subset \R^{\ell}_{\tau}$.
Clearly we have $\mathcal{S}(v) \subset L^2(v) \subset \mathcal{S}'(v)$.
We consider the Laplace transform for these function spaces defined by
 \begin{equation} \label{eqn:def_of_Laplace}
 \mathcal{L} \phi (p) := \int_v e^{i \tau p} \phi(\tau)\,d\tau_1 \wedge \dots \wedge d\tau_{\ell} \qquad (p \in t),
\end{equation}
 where $\phi$ belongs to $\mathcal{S}(v),\,\,L^2(v)$, or $\mathcal{S}'(v)$.
When $\phi \in \mathcal{S}'(v)$,
 the integral should be interpreted in the distribution sense,
 whereas $\mathcal{L} \phi$ is a holomorphic function on the tube domain $t$ eventually.

Let $\mathcal{O}(t)$ denote the space of holomorphic functions on $t$.
We shall introduce some subspaces of $\mathcal{O}(t)$,
 which will turn out to be the Laplace images of function spaces over the cone $v$.  
Let $\mathcal{S}(t)$ be the space of $f \in \mathcal{O}(t)$
 satisfying
\begin{equation} \label{eqn:norm_on_t}
  \sup_{r+i s \in t,\, |\beta| \le b}
   (1 + |r + i s|)^a \Bigl| \Bigl(\frac{\partial}{\partial r}\Bigr)^{\beta}f(r + i s) \Bigr| < +\infty
\end{equation}
 for all $a, b \in \Z_{\ge 0}$.
Let $\mathrm{H}(t)$ be the $L^2$-Hardy space on $t$,
 that is,
 the Hilbert space of $f \in \mathcal{O}(t)$ satisfying
\begin{equation} \label{eqn:def_of_Hardynorm}
\Vert f \Vert^2_{\mathrm{H}(t)} 
 := \sup_{s \in v^*} \int_{\R^{\ell}_r} |f(r + i s)|^2 \,dr_1 \wedge \dots \wedge dr_{\ell} < +\infty.
\end{equation}
Finally,
 let $\mathcal{S}'(v)$ be the space of $f \in \mathcal{O}(t)$
 satisfying
$$ \sup_{r+i s \in \overline{t}} d(s)^a |r + i s|^b |f(r + i s)|< +\infty,$$
 for some $a, b \in \Z_{\ge 0}$, where
$$
d(s) := \begin{cases} \Vert s \Vert & \mbox{ if } \Vert s \Vert<1,\\
1 & \mbox{ if } \Vert s \Vert>1. \end{cases} 
$$
We have the following Paley-Wiener type theorem.

\begin{theorem} \label{thm:Paley-Wiener}
{\rm (i)} A function {(distribution)}
 $f \in \mathcal{S}(t)$ {\rm(}resp. $\mathrm{H}(t)$, $\mathcal{S}'(t)${\rm)}
 has a boundary value as $s \to 0$ in the sense of $\mathcal{S}(\R^{\ell}_{r})$
 {\rm(}resp. $L^2(\R^{\ell}_{r})$, $\mathcal{S}(\R^{\ell}_{r})${\rm)}.\\
{\rm (ii)} The Laplace transform gives an isomorphism from $\mathcal{S}(v)$ 
 {\rm(}resp. $L^2(v)$, $\mathcal{S}'(v)${\rm)}
 onto $\mathcal{S}(t)$ 
 {\rm(}resp. $\mathrm{H}(t)$, $\mathcal{S}'(t)${\rm)}.\\
{\rm (iii)} One has the following commutative diagrams:
$$
\xymatrix{
\mathcal{S}(v) \ar[r]^{\mathcal{L}} \ar[dr]_{\mathcal{F}} & \mathcal{S}(t) \ar[d]_{b} \\
& \mathcal{S}(\R^{\ell}_r),} \quad
\xymatrix{
L^2(v) \ar[r]^{\mathcal{L}} \ar[dr]_{\mathcal{F}} & \mathrm{H}(t) \ar[d]_{b} \\
& L^2(\R^{\ell}_r),} \quad
\xymatrix{
\mathcal{S}'(v) \ar[r]^{\mathcal{L}} \ar[dr]_{\mathcal{F}} & \mathcal{S}'(t) \ar[d]_{b} \\
& \mathcal{S}'(\R^{\ell}_r),} \quad
$$
 where $b$ stands for the map of taking the boundary value, and $\mathcal{F}$ denotes the Fourier transform.
\end{theorem}


Let $\chi_{v}$ denote the characteristic function of the closure of the cone $v$.
The multiplication by $\chi_{v}$ gives a natural projection $\mathcal{S}'(\R^{\ell}_{\tau}) \to \mathcal{S}'(v)$,
 which maps $L^2(\R^{\ell}_{\tau})$ onto $L^2(v)$,
 but the image of $\mathcal{S}(\R^{\ell}_{\tau})$ is not contained in $\mathcal{S}(v)$.
Instead of $\mathcal{S}(v)$,
 we may consider the quotient space $\mathcal{S}([v])$ of $\mathcal{S}(\R^{\ell}_{\tau})$ 
 over the space consisting of functions whose support is contained in $\R^{\ell}_{\tau} \setminus v$.
Then $\mathcal{S}'(v)$ is the dual of $\mathcal{S}([v])$,
 and we have a canonical projection $\mathcal{S}(\R^{\ell}_{\tau}) \to \mathcal{S}([v])$.

\subsection{Cauchy-Szeg\"o projection} 
We shall introduce a specific convolution operator $\mathcal{K}_v : \mathcal{S}'(\R^{\ell}_r) \to \mathcal{S}'(t)$
 called the Cauchy-Szeg\"o projection,
 which is the Fourier dual of the projection $\mathcal{S}'(\R^{\ell}_{\tau}) \to \mathcal{S}'(v)$
 mentioned in the previous subsection.
{Apparently,}
 $\mathcal{K}_v$ gives left inverses of the boundary value maps $b$
 in Theorem \ref{thm:Paley-Wiener}.

Let $K_v := \mathcal{L}(\chi_v) \in \mathcal{S}'(t)$,
 and define the operator $\mathcal{K}_v$ by
$$
 \mathcal{K}_v F := K_v * F \in \mathcal{S}'(t) \qquad (F \in \mathcal{S}'(\R^{\ell}_r)),
$$
 where the convolution is interpreted in the sense of distributions.
We refer to \cite{GV92} for the detailed study of convolution operators.
{Note that} $K_v$ is expressed as
$$
K_v(p) = \int_v e^{i \tau p} d\tau_1 \wedge \dots \wedge d\tau_{\ell} \qquad (p \in t),
 $$
{and that, if $F \in L^2(\R^{\ell}_r)$, we have
$$
 \mathcal{K}_v F(p) = \int_{\R^{\ell}_r} K_v(p-r)F(r)\,dr_1 \wedge \dots \wedge dr_{\ell}. 
$$
}  Since the multiplication by $\chi_v$ is just the identity operator on $\mathcal{S}'(v)$ and $L^2(v)$,
 we have the following theorem. 

\begin{theorem} \label{thm:Cauchy-Szegoe}
The operator $\mathcal{K}_v$ gives left inverses of the boundary value maps
 $b : \mathcal{S}'(v) \to \mathcal{S}(\R^{\ell}_r)$
 and $b : H(t) \to L^2(\R^{\ell}_r)$
 in Theorem~\ref{thm:Paley-Wiener} (iii).
Namely,
 one has $\mathcal{K}_v \circ b (f) = f$ for $f \in \mathcal{S}'(t)$, and $\mathcal{K}_v F \in \mathrm{H}(t)$ if $F \in L^2(t)$.
\end{theorem}



%
%
%
%
%
\section{Geometric duality for non-convex cones} 
\subsection{Slices and wedges of a non-convex cone}
Let $V$ be an open non-convex cone in the vector space $\R^n_{\xi}$ of $n$-dimensional row vectors $\xi$.
\textit{The index of non-convexity} of the cone $V$ is
 the minimum of the codimension of a linear subspace $L$ of $\R^n_{\xi}$
 such that one of the connected components $v$ of the intersection $L \cap V$ is a pointed convex cone.
Such a pointed cone $v$ of minimal codimension is called \textit{a slice} of $V$. 
In what follows, we denote by $q$ the index of non-convexity of $V$, and by $\ell$ the dimension of the slices of $V$.
Clearly, $q = n - \ell$.

Let $\R^n_x$ be the space of $n$-dimensional column vectors $x$,
 which can be regarded as the dual space of $\R^n_{\xi}$ by the coupling
$\xi x = \xi_1 x_1 + \dots + \xi_n x_n$.
For a slice $v \subset V \cap L$, 
 \textit{the wedge} $V^*(v) \subset \R^n_x$ dual to $v$ is defined by
$$
 V^*(v) := \set{x \in \R^n_x}{\xi x >0 \quad (\forall \,\xi \in v \setminus \{0\})}.
$$
Then $V^*(v)$ is an open convex cone in $\R^n_x$, but it is not pointed.

We consider the space $\R^n_{\xi}$ with the orientation
 and the cone $V$ with the induced orientation. 
Then slices are cones with an orientation.
We consider oriented union of slices.
In such union, a compensation can appear as a result of orientation.
The unions over cycles $\gamma$ are the cone $V$ with an integral coefficient $c(\gamma)$
 which can be $0$.
A cycle $\gamma$ of dimension $q$ is called \textit{generating} if this coefficient is not equal to $0$.
Under a small perturbation, a generating cycle continues to be generating.
In many cases (see examples below)
 we can choose a generating cycle from non-intersecting slices, 
 but it is not necessary.
And at the perturbation,
 the property is not preserved.

Then we put the condition on $V$ that there is a set of slices that contains a generating cycle.
We call a set of slices \textit{admissible} if it contains at least one generating cycle.
Let us fix an admissible set, and for this set, define $V^*$ as
 the oriented union of $V^*(v)$ over the admissible set.

In many cases including symmetric homogeneous cones,
 $V^*$ is independent of a choice of admissible set,
 whereas only the covering by $V^*(v)$ will be different.
We do not consider this problem.

\subsection{Non-convex cones defined by quadratic forms}
Let $q$ be an integer such that $1 \le q \le n-1$, and put $\ell := n-q$.
{Let us define 
$$
V_{q+1, \ell-1} := \set{ \xi \in \R^n_{\xi}}{F_{q+1, \ell-1}(\xi) :=\xi_1^2 + \dots + \xi_{q+1}^2 - \xi^2_{q+2} - \dots - \xi^2_n >0},
$$
which is a subset of the space $\R^n_{\xi}$ of $n$-dimensional row vectors as in Subsection 3.1.} 
Namely, $V_{q+1, \ell-1}$ is the positivity set for the standard quadratic form $F_{q+1, \ell-1}$ {on $\R^n_{\xi}$} of signature $(q+1,\ell-1)$.
Then $V_{q+1, \ell-1}$
 is an open connected non-convex cone,
 which is invariant under the linear action of $SO(q+1, \ell-1)$
 given by the right-multiplication:
 if $\xi \in V_{q+1, \ell-1}$, then 
 $\xi g \in V_{q+1, \ell-1}$ for $g \in SO(q+1, \ell-1)$.
The cone $V_{q+1, \ell-1}$ is symmetric (\cite[Section 2]{FG96}).

Let $L$ be a subspace of $\R^n_{\xi}$, and $(a,b)$ be the signature of the restriction of the form $F_{q+1, \ell-1}$ to $L$.
Then the intersection $L \cap V_{q+1, \ell-1}$ has a pointed connected component if and only if $a=1$.
Thus, the dimension of such $L$ is maximal if $(a,b) = (1, \ell-1)$.
Let $S$ be the submanifold of the Grassmannian consisting of $\ell$-dimensional space $L \subset \R^n_{\xi}$
 for which the restriction $F_{q+1, \ell-1}|_L$ is of signature $(1, \ell-1)$. 
Clearly $S$ is invariant under the action of $SO(q+1, \ell-1)$.
For each $L \in S$,
 the intersection $L \cap V_{q+1, \ell-1}$ has two connected components which are both pointed cones, that is, slices of $V_{q+1, \ell-1}$.
Therefore the set of all slices forms a double covering of $S$,
 and it is also $SO(q+1, \ell-1)$-invariant. 
  
Define $L_0 := \set{\xi \in \R^n_{\xi}}{\xi_1 = \dots = \xi_q = 0}$.
Then $L_0 \in S$,
 and we have $L_0 \cap V_{q+1, \ell-1} = v_0 \cup (-v_0)$,
 where $v_0$ is the slice given by
$$
 v_0 := \set{\xi \in L_0}{\xi_{q+1}>0,\,\,\xi_{q+1}^2 - \xi_{q+2}^2 - \dots - \xi_n^2 >0 }.
$$
We regard $SO(q+1)$ as a subgroup of $SO(q+1, \ell-1)$ in a natural way. 
Let $\gamma_0$ be the cycle consisting of slices $v_0 k$ with $k \in SO(q+1)$. 
Then $\gamma_0$ is $q$-dimensional,
 and it is a generating cycle of $V_{q+1, \ell-1}$.

The wedge $V^*(v_0) \subset \R^n_x$ dual to $v_0$ is described as
$$
 V^*(v_0) = \set{x \in \R^n_x}{x_{q+1}>0,\, x_{q+1}^2 - x_{q+2}^2 - \dots - x_n^2 >0},
$$  
 which is linearly isomorphic to a direct product of $\R^{q}$ and the $\ell$-dimensional Lorentz cone.
For a slice $v = v_0 g$ with $g \in SO(q+1, \ell-1)$, we have $V^*(v) = g^{-1} V^*(v)$.
Therefore,
 for the set of all slices or any admissible subset,
 the dual cone $V^* = \bigcup\limits_v V^*(v)$ coincides with $V_{q+1, \ell-1}$.

\subsection{Non-convex cone of square matrices of positive determinant}
Let $M(m, \R)$ be the vector space of real $m \times m$ matrices,
 and $M_+(m)$ the subset of $M(m,\R)$ consisting of $X \in M(m,\R)$ for which $\det X >0$.
We assume that $m \ge 2$.
Then $M_+(m)$ is an open connected non-convex cone.
Let $\mathrm{Sym}$
 be the space of $m \times m$ symmetric matrices.
Then the identity component of $\mathrm{Sym} \cap M_+(m)$ containing $I_m$ is equal to 
 the set $v_0$ of positive definite symmetric matrices,
 because the determinant is zero on the boundary of $v_0$.
It is well known that $v_0$ is a pointed convex cone. 
Moreover, $v_0$ is a slice of $M_+(m)$ with codimension $q := m(m-1)/2$.

Let $G$ denote $GL_+(m,\R)$.
The group $G \times G$ acts on $M_+(m)$ by
$$
 (g_1,\, g_2)\cdot X := g_1 X g_2^T \quad (X \in M_+(m),\,\,(g_1,\,g_2) \in G \times G).
$$
By this action, 
 $G \times G$ acts on the set of slices transitively,
 where the isotropy at $v_0$ is $\Delta G = \set{(g,g)}{g \in G}$.
Then we have the convenient parametrization of the slices as $v(g) := v_0 g$ with $g \in G$.
We do not discuss further why there is no other slices.

Let $\gamma_0$ be the cycle consisting of slices $v(u)$ with $u \in SO(m)$.
Then $\gamma_0$ is $q$-dimensional
 (recall $q = m(m-1)/2$), and it is a generating cycle of the cone $M_+(m)$.
Indeed,
 we have $M_+(m) = \bigsqcup_{u \in SO(m)} v(u)$ (disjoint union),
 which follows from the polar decomposition of matrices: 
 for $X \in M_+(m)$,
 there exist unique $Y \in v_0$ and $u \in SO(m)$ for which $X = Yu$.

We define the inner product on the vector space $M(m, \R)$ by 
 $(X|Y) := \mathrm{tr}\,(X Y^T)$.
The wedge $V^*(v_0) \subset M(m,\R)$ dual to $v_0$ with respect to this inner product is described as
\begin{align*}
V^*(v_0) &= \set{X \in M(m,\R)}{X + X^T \in v_0}\\
 &= \set{Y + A}{Y \in v_0,\, A = - A^T}.
\end{align*}
Then $V^*(v_0)$ is contained in $M_+(m)$,
 which is observed in \cite[Section 1.5]{DG93}, and also shown in \cite[Proposition 4.2 (i)]{FG96}
 as a statement for general non-convex symmetric cones in terms of Jordan algebras.
Here we give a direct proof.
Take $X = Y + A \in V^*(v_0)$ with $Y \in v_0$ and $A = - A^T$,
 and put $H := -i A$. 
{Then $Y$ and $H$ are Hermitian matrices,
 and $Y$ is positive definite.
We see that
 $\det X = \det (Y + i H)$ is non-zero by an argument of linear algebra similar to the classical Siegel's lemma.
Namely, 
 since there exists $P \in GL(m, \R)$ for which $PYP^{-1} = I_m$ and $PHP^{-1}$ equals a real diagonal matrix $D$,
 we obtain $\det X = \det (I_m + i D) \ne 0$.
We remark that a general form of Siegel's lemma is given in \cite[Section 4]{FG96}.
} 
Since $\det X$ is real, we have $\det X>0$
 by the connectedness of $M_+(m)$.
Therefore we have
 $V^*(v_0) \subset M_+(m)$. 
By the $SO(m)$-invariance of the trace inner product,
 we have $V^*(v(u)) = u^{-1} V^*(v_0)$ for $u \in SO(m)$,
 so that $M_+(m) = \bigcup_{u \in SO(m)} V^*(v(u))$. 
Hence the dual cone $V^*$ coincides with $M_+(m)$ for the special cycle $\gamma_0$ as well as the set of all slices.

\subsection{Parameter space of slices}

Let $V \subset \R^n_{\xi}$ be an open non-convex cone,
 $q$ the index of non-convexity of $V$,
 and $\ell = n - q$ the dimension of the slices of $V$.  
Let us fix an admissible set of slices of $V$,
{whereas we will work with not the set of slices but the frames of them.
Accordingly, the parameter space will be a subset of the Stiefel manifold $\mathrm{St}(\ell, n)$,
 that is, the set of all $\ell$-frames in $\R^n_{\xi}$.
An element of $\mathrm{St}(\ell, n)$ is identified with an $\ell \times n$ real matrix 
 $$\Xi = \begin{pmatrix} \xi^1 \\ \xi^2 \\ \vdots \\ \xi^{\ell} \end{pmatrix}
 \in \mathrm{Mat}(\ell,n,\R)
 \qquad (\xi^1,\,\xi^2, \dots, \xi^{\ell} \in \R^n_{\xi})$$
 satisfying $\mathrm{rank}\, \Xi = \ell$, 
 where the row vectors $\xi^1, \dots, \xi^{\ell}$ form an $\ell$-frame in $\R^n_{\xi}$.}

{Let $\Xi(V) \subset \mathrm{St}(\ell, n)$ be the set of $\Xi$  
 for which the positive $\ell$-hedorn spanned by $\xi^1,\,\xi^2, \dots, \xi^{\ell}$
 is contained in a slice $v(\Xi)$ belonging to the fixed admissible set.}
We call $\Xi(V)$ \textit{the parameter space} associated to the admissible set of slices. 
Using $\Xi \in \Xi(V)$,
 we shall introduce convenient coordinates on the slice $v(\Xi)$ and related dual objects.

For $\Xi \in {\mathrm{St}(\ell, n)}$ and a row vector $\tau \in \R^{\ell}_{\tau}$,
 we have
 \begin{equation} \label{eqn:Xi}
 \tau \Xi = \tau_1 \xi^1 + \dots + \tau_{\ell} \xi^{\ell} \in \R^n_{\xi}.
 \end{equation}
Let $\R^{\ell}_{\tau}(\Xi)$ be the subspace $\set{\tau\Xi}{\tau \in \R^{\ell}_{\tau}}$ of $\R^n_{\xi}$.
Namely, $\R^{\ell}_{\tau}(\Xi)$ is the image of the imbedding $\R^{\ell}_{\tau} \owns \tau \mapsto \tau \Xi \in \R^n_{\xi}$.
By (\ref{eqn:Xi}), the space $\R^{\ell}_{\tau}(\Xi)$ is spanned by $\xi^1, \dots, \xi^{\ell}$.
The row vector $\tau$ is used as a coordinate of $\R^{\ell}_{\tau}(\Xi)$. 

We denote by $\R^{\ell}_{r}(\Xi)$
 the dual space of $\R^{\ell}_{\tau}(\Xi)$.
The space
 $\R^{\ell}_{r}(\Xi)$
 is realized as the quotient space of $\R^n_x$ with $(\R^{\ell}_{\tau}(\Xi))^{\perp}$.
We observe
 \begin{align*}
 (\R^{\ell}_{\tau}(\Xi))^{\perp} 
 &= \set{x \in \R^n_x}{\tau \Xi x = 0 \quad (\forall \tau \in \R^r_{\tau})}\\
 &= \set{x \in \R^n_x}{\Xi x = 0} 
 = \mathrm{Ker}\, \Xi.
\end{align*}
Thus the $\ell$-dimensional column vector $r = \Xi x\,\,\,(x \in \R^n_x)$
 gives a coordinate of $\R^{\ell}_r(\Xi)$.
Indeed, 
 the coupling of $\tau \Xi \in \R^{\ell}_{\tau}(\Xi)$ and $r = \Xi x$ equals
 $\tau r = \tau_1 r_1 + \dots + \tau_{\ell} r_{\ell}$.

{For $\Xi \in \Xi(V)$,}
 let $v^*(\Xi)$ be the dual cone of $v(\Xi)$ in the space $\R^{\ell}_r(\Xi)$.
Then
$$
 v^*(\Xi) =
 \set{r  = \Xi x}
     {\tau r > 0 \quad
      (\forall \tau \Xi \in \overline{v(\Xi)} \setminus \{0\})}.
$$
Let us denote by $V^*(\Xi)$ the wedge dual to $v(\Xi)$.
Then 
\begin{equation} \label{eqn:def_of_VXi}
 V^*(\Xi) = \set{x \in \R^n_x}{\xi x > 0 \quad (\forall \xi \in \overline{v(\Xi)} \setminus \{0\}) }.
\end{equation}
The relation between $v^*(\Xi)$ and the wedge $V^*(\Xi)$ is given by 
\begin{equation} \label{eqn:def_of_vXistar}
 v^*(\Xi) = \set{\Xi x}{x \in V^*(\Xi)}.
\end{equation}
The fibers of the projection $V^*(\Xi) \owns x \mapsto \Xi x \in v^*(\Xi)$
 are affine spaces in $\R^n_{x}$ parallel to 
 the vector space $\mathrm{Ker}\,\Xi = \set{x}{\Xi x = 0}$.
In other words,
 a wedge is linearly isomorphic to the direct product of a pointed convex cone and a real vector space.

Let $T$ and $T(\Xi)$ be the tube domains $\R^n_x + i V^*$ and $\R^n_x + i V^*(\Xi)$ respectively,
 where $\Xi \in \Xi(V)$.
Then we have $T= \bigcup_{\Xi \in \Xi(V)} T(\Xi)$.
In other words,
 the system $\{T(\Xi)\}_{\Xi \in \Xi(V)}$
 gives a Stein covering of the tube domain $T$.
Let $t(\Xi)$ be the tube $\R^{\ell}_r(\Xi) + i v^*(t(\Xi))$.
Then we have a projection $\pi_{\Xi}: T(\Xi) \owns z \mapsto p = \Xi z \in t(\Xi)$.

\section{Cohomological Laplace transform on $\mathcal{S}(V)$}
Let $V$ be a non-convex cone satisfying the conditions of Section 3
 with a fixed admissible set of slices,
 and $\Xi(V)$ the corresponding set of $\ell$-frames of slices.

In this section,
 we define the cohomological Laplace transform $\mathcal{L}f$ for $\mathcal{S}(V)$:
 at the beginning as a \u{C}ech cocycle for the covering {$\{T(\Xi)\}_{\Xi \in \Xi(V)}$ of $T$}.
Our version of \u{C}ech cohomology for smoothly parameterized \u{C}ech covering has a specific {form}
 in which a usual \u{C}ech condition of cocycle is replaced by a system of differential equations
 {(John system)}
 which appeared in integral geometry.
Following conception {from \cite{G13}},
 we give an equivalent representation of the cohomology,
 {that is,}
 a closed form on $\Xi(V)$ depending holomorphically on parameters
 at {the tubes $T(\Xi)$ from the covering.}
We define this differential form in two ways:
 as a result of a differential operator from the \u{C}ech cocycle, and {as an image of a direct operator from $f$.}
We define an operator of {boundary values at $\mathcal{S}(\R^n_x)$} on the space of cohomology $\mathcal{S}(T)$
 and find that the boundary values of the cohomological Laplace transform coincides with the Fourier transform
 of the initial function. 
Let us remind that
 it was stated in Introduction as the principal request for the definition of Laplace transform
 for non-convex case.
{We recall that, for $f \in \mathcal{S}(\R^n_{\xi})$, 
 the Fourier transform is denoted by $\mathcal{F} f$, that is,
$$
 \mathcal{F} f(x) := \int_{\R^n_{\xi}} e^{i \xi x} f(\xi) d\xi_1 \wedge \dots d\xi_n \qquad (x \in \R^n_x).
$$
}We connect the cohomological Laplace transform
 with the Radon-John transform of $\mathcal{F} f$.
It is connected with the reconstruction $\mathcal{L}f$ through $\mathcal{F} f$:
 cohomological version of the Cauchy-Szeg\"o operator.

\subsection{Cohomological Laplace transform (\u{C}ech version)}
Similarly to Section 1, 
 we denote by $\mathcal{S}(V)$ (resp. $L^2(V)$)
 the subspace of $\mathcal{S}(\R^n_{\xi})$ (resp. $L^2(\R^n_{\xi})$)
 consisting of functions
 whose support is contained in the closure $\overline{V}$ of the non-convex cone $V$.

For $f \in \mathcal{S}(V)$,
 we define 
\begin{equation} \label{eqn:def_of_cohL}
\begin{aligned}
 \mathcal{L}f(\Xi,z) := \int_{ \{\tau\,;\, \tau \Xi \, \in v(\Xi)\} } e^{i \tau \Xi z} f(\tau \Xi) d\tau_1\wedge \dots \wedge d\tau_{\ell}\\
 \qquad (\Xi \in \Xi(V),\,\,z \in T(\Xi)),
\end{aligned}
\end{equation}
 which we call \textit{the cohomological Laplace transform} of $f$.
{Then $\mathcal{L}f$ is the holomorphic extension of the Fourier transform of the restriction of $f$ on the slices,
 which depends, of course, on the choice of $\ell$-frame $\Xi$ at the slice.
Moreover,
 $\mathcal{L}f(\Xi,z)$ is constant on affine spaces $\set{ z \in \C^n_z}{\Xi z = p} \subset T(\Xi)$ for $p \in t(\Xi)$
 by definition.
Indeed, introducing the {``small''} Laplace transform
 \begin{equation} \label{eqn:cohL_for_p}
 \begin{aligned}
 \mathcal{L}_p f(\Xi,p) := \int_{ \{\tau\,;\, \tau \Xi \, \in v(\Xi)\} } e^{i \tau p} f(\tau \Xi) d\tau_1\wedge \dots \wedge d\tau_{\ell}\\
 \qquad (\Xi \in \Xi(V),\,\,p \in t(\Xi)),
 \end{aligned}
\end{equation}
 we see that $\mathcal{L}f(\Xi, \cdot)$ is the pullback of $\mathcal{L}_p f(\Xi, \cdot)$ by the projection
 $\pi_{\Xi}:T(\Xi) \owns z \mapsto p = \Xi z \in t(\Xi)$.
On the other hand,
 $\mathcal{L}_p f(\Xi,\, \cdot) \in \mathcal{S}(t(\Xi))$ by Theorem \ref{thm:Paley-Wiener}.
Therefore the function $\mathcal{L}f(\Xi,\,\cdot)$ on $T(\Xi)$ belongs to 
 the pullback $\pi^*_{\Xi} \mathcal{S}(t(\Xi)) := \set{\phi \circ \pi_{\Xi}}{\phi \in \mathcal{S}(t(\Xi))}$
 of the function space $\mathcal{S}(t(\Xi))$ for each $\Xi \in \Xi(V)$.}

{Let us consider differential equations that $\mathcal{L}f(\Xi,z)$ satisfies with respect to the variable $\Xi$.}
We call \textit{John system} the following system of differential equations for a function $\psi$ of $\Xi$:
\begin{equation} \label{eqn:John-sys}
 \frac{\partial^2}{\partial \xi^j_{\alpha} \partial \xi^k_{\beta}} \psi(\Xi) 
= \frac{\partial^2}{\partial \xi^j_{\beta} \partial \xi^k_{\alpha}} \psi(\Xi) \qquad 
\left( \begin{aligned} j,k &=1, \dots, \ell\\ \alpha, \beta &= 1, \dots, n \end{aligned} \right).
\end{equation}


\begin{theorem} \label{thm:Cech-cocycle}
The cohomological Laplace transform $\mathcal{L}f(\Xi,z)$ is a solution of John system {with respect to $\Xi$.}
\end{theorem}

\begin{proof}
{Let $U$ be an open subset of $\Xi(V)$},
 and define $\tilde{U} := \bigcup_{\Xi \in U} v(\Xi)$.
We assume that $U$ is sufficiently small so that
 $\tilde{U} \cap \R^{\ell}_{\tau}(\Xi) = v(\Xi)$ for $\Xi \in U$,
 which means that $\R^{\ell}_{\tau}(\Xi) \ne \R^{\ell}_{\tau}(\Xi')$ for distinct $\Xi, \Xi' \in U$.
Noting that $\tilde{U}$ is an open subset of $\R^n_{\xi}$,
 we put $f_{\tilde{U}}(\xi) := f(\xi) \chi^{}_{\tilde{U}}(\xi)$.
Then, by (\ref{eqn:def_of_cohL}) we have for $\Xi \in U$
 $$
 \mathcal{L}f(\Xi,z) = \int_{\R^{\ell}_{\tau}} e^{i \tau  \Xi z} f_{\tilde{U}}(\tau \Xi) d\tau_1\wedge \dots \wedge d\tau_{\ell}.
 $$
Therefore,
\begin{align*}
{}& \frac{\partial^2}{\partial \xi^j_{\alpha} \xi^k_{\beta}} \mathcal{L}f(\Xi;p)
  = \frac{\partial^2}{\partial \xi^j_{\beta} \xi^k_{\alpha}} \mathcal{L}f(\Xi;p)\\
 &= \int_{\R^{\ell}_{\tau}} e^{i \tau p} 
   \tau_j \tau_k 
  \Bigl( i z_{\alpha} + \frac{\partial}{\partial \xi_{\alpha}}\Bigr) 
  \Bigl( i z_{\beta} + \frac{\partial}{\partial \xi_{\beta}}\Bigr) 
 f_{\tilde{U}}(\tau \Xi)\,d\tau_1\wedge  \dots \wedge d\tau_{\ell},
\end{align*}
 which verifies the statement.
\end{proof}

 
\begin{remark}
At this situation,
 it is possible to prove that the John system is equivalent to \u{C}ech condition on cocycles.
\end{remark}  

Let us consider a special class of solutions $F(\Xi,z)$ of John system
 with respect to $\Xi$ with a holomorphic parameter $z \in T(\Xi)$.
We assume that $F(\Xi, \cdot)$ belongs to $\pi_{\Xi}^*\mathcal{S}(t(\Xi))$ for each $\Xi \in \Xi(V)$.
In other words,
 there exists a function $\phi(\Xi, p)\,\,\,(p \in t(\Xi))$ such that $\phi(\Xi,\, \cdot) \in \mathcal{S}(t(\Xi))$ and $F(\Xi, z) = \phi(\Xi, \Xi z)$
 for each $\Xi \in \Xi(V)$.
We define $\mathcal{S}(T)$ to be the space of $F$ such that $\phi(\Xi, \cdot)$ lies at $\mathcal{S}(t(\Xi))$ uniformly for $\Xi \in \Xi(V)$,
 which means that the left-hand side of (\ref{eqn:norm_on_t}) with $f = \phi(\Xi, \cdot)$ and $t = t(\Xi)$ is bounded by a constant independent of $\Xi$. 


\begin{theorem} \label{thm:Image_of_cohL}
The cohomological Laplace transform gives an isomorphism from $\mathcal{S}(V)$ onto $\mathcal{S}(T)$.
\end{theorem}

From the argument preceding Theorem~\ref{thm:Image_of_cohL},
 we see that the image of cohomological Laplace transform of $\mathcal{S}(V)$ is contained in $\mathcal{S}(T)$.
The converse inclusion will be shown later {(Theorem~\ref{thm:Szegoe} (ii))}.

\begin{remark}
Of course we can inverse $F(\Xi, z)$ in Theorem~\ref{thm:Image_of_cohL} for each $\Xi$
 and obtain a function $f(\Xi;\xi)$ on slices $v(\Xi)$.
We need to verify a compatibility that all these functions are independent of $\Xi$ 
 and are restrictions of one function $f(\xi)$ to $v(\Xi)$. 
\end{remark}

So we realize $\mathcal{L}f$ as special \u{C}ech cocycles for the covering $\{T(\Xi)\}$.

\subsection{Realization of cohomological {Laplace} transform by a closed form on the parameter space with holomorphic parameters}
In this subsection,
 we give the form of cohomological Laplace transform
 using another language for the representation of $\bar{\partial}$-cohomology at $T$:
 closed differential form $\psi(z;\Xi, d\Xi)$ on the parameter space $\Xi(V)$ of the covering $\{T(\Xi)\}$ by tubes,
 where $\psi(z;\Xi,d\Xi)$ 
 depends on the parameter $z \in T(\Xi)$ holomorphically for each $\Xi \in \Xi(V)$
 {(cf. \cite{BEG05})}.
To obtain such form-transform $\tilde{\mathcal{L}}f$, 
 we integrate the form-integrand at the Fourier transform of $f$ on {different infinitesimal neighbourhoods} of the slice $v(\Xi)$.
The result is a $q$-form {on transversal variables $\Xi \in \Xi(V)$.} 
Technically it is the direct image. 
Namely, introducing the map
\begin{equation} \label{eqn:def_of_Phi}
 \Phi: \R^{\ell}_{\tau} \times \Xi(V) \owns (\tau, \Xi) \mapsto \tau \Xi \in \R^n_{\xi},
\end{equation}
 we define
\begin{equation} \label{eqn:def_of_formLf}
\begin{aligned}
 \tilde{\mathcal{L}}f(z;\Xi, d\Xi) 
 := \int_{ \{\tau \,;\,\tau \Xi\, \in v(\Xi)\}}
 \Phi^*(e^{i \xi z} f(\xi)\,d\xi_1 \wedge \dots \wedge d\xi_n)\\
\qquad (\Xi \in \Xi(V),\,z \in T(\Xi)),
\end{aligned}
\end{equation}
 which we call \textit{the cohomological Laplace form-transform} of $f$.
Since we take the direct image of the closed form $\Phi^*(e^{i \xi z} f(\xi)\,d\xi_1 \wedge \dots \wedge d\xi_n)$,
 we obtain the following.


\begin{proposition} \label{prop:cosed_formLf}
The cohomological Laplace form-transform $\tilde{\mathcal{L}}f$ is a closed form on $\Xi(V)$
 {with parameters $z$}.
\end{proposition}


\begin{remark}
{\rm (i)}
Domains of parameters $z$ are different for different $\Xi$,
 but the notion of closedness of such forms is well-defined.\\
{\rm (ii)}
Let us suppose that for every $z \in T$
 the set 
 $$\Xi(z) := \set{\Xi \in \Xi(V)}{z \in T(\Xi)}$$
 is contractible.
In this situation,
 complex of forms such as $\psi(z;\Xi, d\Xi)$ and the corresponding cohomology
 are considered in \cite{BEG05}.
{In that paper, it is proved under the contractibility condition that
 such cohomology is isomorphic to Dolbeault cohomology. }
It justifies the name ``cohomology'' for our transform.
{Let us remind this construction. 
Consider the projection $(\Xi,z) \mapsto z,\,\,z \in T(\Xi)$. 
Then $\Xi(z)$ are fibers of this projection. 
Let us fix any global section $\Gamma$ of this fibering over $T$. 
The condition of contractibility guaranties the existence of $\Gamma$. 
Restrict forms $\psi$ on $\Gamma$ and consider them as forms on $T$ and take their $(0,q)$-parts which will be $\bar{\partial}$-closed. 
We have the operator at Dolbeault cohomology $H^{(0,q)}(T,\mathcal{O})$,
 and can realize cohomological Laplace transform at Doulbeault picture.}
\end{remark}

\subsection{Connections between two cohomological Laplace transforms}
In this subsection,
 we shall investigate a connection between two Laplace transforms $\mathcal{L}f(\Xi,z)$ and $\tilde{\mathcal{L}}f(z;\Xi,d\Xi)$.
For this, we utilize determinants of matrices containing rows of $1$-forms,
 where we apply exterior product in the multiplication, and expand the determinant from top to bottom
 (cf. \cite{G13}).
Then the same row can repeat several times in a non-zero determinant.
The notation $a^{\{k\}}$ means that the row $a$ repeats $k$ times.
For instance,
 writing $d\xi$ for the row vector $(d\xi_1, d\xi_2, \dots, d\xi_n)$ whose entries are $1$-forms on the vector space $\R^n_{\xi}$,
 we have
 $$ d\xi_1 \wedge d\xi_2 \wedge \dots \wedge d\xi_n = \frac{1}{n!} \det(d\xi^{\{n\}}).$$
Now we consider the $n$-form
\begin{align*}
 \Phi^* (e^{i \xi z} f(\xi)\,d\xi_1 \wedge \dots \wedge d\xi_n)
 &= (n!)^{-1} e^{i \tau \Xi z} f(\tau \Xi)\,\det((\Phi^*d\xi)^{\{n\}}).
\end{align*}
By (\ref{eqn:def_of_Phi}),
 we have 
$$
 \Phi^* d\xi = \sum_{j=1}^{\ell} (\xi^j d\tau_j + \tau_j d\xi^j), 
$$
 where 
\begin{align*}
 \xi^j d\tau_j &= (\xi^j_1 d\tau_j, \xi^j_2 d\tau_j, \dots, \xi^j_n d\tau_j),\\
 \tau_j d\xi^j &= (\tau_j d\xi^j_1, \tau_j d\xi^j_2, \dots, \tau_j d\xi^j_n).
\end{align*}
By a combinatorial argument, we have
\begin{align*}
 {}&\frac{1}{n!}\det((\Phi^* d\xi)^{\{n\}})\\
 &= \frac{1}{\ell! (n-\ell)!}
   d\tau_1 \wedge d\tau_2 \wedge \dots \wedge d\tau_{\ell} \wedge \det\begin{bmatrix}\Xi \\ (\sum_{j=1}^{\ell} \tau_j d\xi^j)^{ \{n-\ell \} } \end{bmatrix}\\
 &\qquad \qquad + (\mbox{lower terms of $d\tau_j$'s}).
\end{align*}
Thus, recalling $q = n- \ell$,
 we have by (\ref{eqn:def_of_formLf})
\begin{equation} \label{eqn:Lf-det}
\begin{aligned}
{}& \tilde{\mathcal{L}}f(z;\Xi, d\Xi) \\
 &= \frac{1}{\ell! q!} 
   \int_{ \{\tau \,;\,\tau \Xi\, \in v(\Xi)\}}
    e^{i \tau \Xi z} f(\tau \Xi)\, d\tau_1 \wedge \dots \wedge d\tau_{\ell} 
   \wedge \det\begin{bmatrix}\Xi \\ (\sum_{j=1}^{\ell} \tau_j d\xi^j)^{ \{ q \} } \end{bmatrix}.
\end{aligned}
\end{equation}
On the other hand,
 we observe that
\begin{equation} \label{eqn:lambda_m}
\det\begin{bmatrix}\Xi \\ (\sum_{j=1}^{\ell} \tau_j d\xi^j)^{ \{ q \} } \end{bmatrix}
= \sum_{\substack{ m = (m_1, \dots, m_{\ell})  \\ |m| = q}} \frac{q !}{m_1 ! \dots m_{\ell} !}\,
 \tau^m \det\begin{bmatrix}\Xi \\ d\Xi^{ [ m ] } \end{bmatrix},
\end{equation}
where $|m| := m_1 + \dots + m_{\ell},\,\,\tau^m := \tau_1^{m_1} \dots \tau_{\ell}^{m_{\ell}}$,
 and $d\Xi^{[m]}$ denotes the $q \times n$ matrix whose first $m_1$-rows are $d\xi^1$,
 the next $m_2$-rows are $d\xi^2$,... and so on.
We put
$$
 \lambda_m(\Xi, d\Xi) := \frac{1}{\ell! m_1 ! \dots m_{\ell} !}\det\begin{bmatrix}\Xi \\ d\Xi^{ [ m ] } \end{bmatrix}.
$$
Then
 (\ref{eqn:Lf-det}) is rewritten as
\begin{equation} \label{eqn:decomp-pLaplace}
\begin{aligned}
{}&\tilde{\mathcal{L}}f(z;\Xi, d\Xi) \\
&= \sum_{|m| = q}
 \left\{ \int_{ \{\tau\,;\, \tau \Xi \, \in v(\Xi)\} } e^{i \tau  \Xi z} f(\tau \Xi) \tau^m d\tau_1\wedge \dots \wedge d\tau_{\ell}\right\} 
 \lambda_m(\Xi, d\Xi),
\end{aligned}
\end{equation}
 which together with (\ref{eqn:cohL_for_p}) yields
\begin{equation} \label{eqn:pre-kappa}
\tilde{\mathcal{L}}f(z; \Xi, d\Xi) 
 = \sum_{|m| = q}
  \Bigl(\frac{1}{i} \frac{\partial}{\partial p}\Bigr)^m {\mathcal{L}_p f(\Xi,p)|_{p = \Xi z}} \lambda_m(\Xi, d\Xi),
 \end{equation}
 where $(\frac{1}{i} \frac{\partial}{\partial p})^m = \prod_{j=1}^{\ell}(\frac{1}{i} \frac{\partial}{\partial p_j})^{m_j}$.

{
Let $F(\Xi,z) \in \mathcal{S}(T)$.
Then there exists a function $\phi(\Xi,p)$ such that $\phi(\Xi, \cdot) \in \mathcal{S}(t(\Xi))$ and $F(\Xi, z) = \phi(\Xi, \Xi z)$ for $z \in T(\Xi)$. 
Let $\kappa$ denote the differential operator defined on $\mathcal{S}(T)$ by
 \begin{equation} \label{eqn:def_of_kappa}
\begin{aligned}
\kappa F(\Xi, z) &:=  \frac{1}{\ell! q!}
 \det\begin{bmatrix}\Xi \\ (\sum_{j=1}^{\ell} (\frac{1}{i} \frac{\partial}{\partial p_j}) \otimes d\xi^j)^{ \{ q \} } \end{bmatrix}
 \phi(\Xi, p)|_{p = \Xi z}\\
 &= \sum_{|m| = q} 
 \Bigl(\frac{1}{i} \frac{\partial}{\partial p} \Bigr)^m \phi(\Xi, p)|_{p= \Xi z} \lambda_m(\Xi; d\Xi).
 \end{aligned}
 \end{equation}
Concluding the argument above,
 we get the following formula connecting two Laplace transforms.}

\begin{theorem} \label{thm:cohL-formL}
For $f \in \mathcal{S}(V)$, one has
$$ \tilde{\mathcal{L}}f(z;\Xi, d\Xi) = \kappa \mathcal{L}f(\Xi,z)
 \qquad (\Xi \in \Xi(V),\,\,z \in T(\Xi)).$$
\end{theorem}

{
Let $\tilde{\mathcal{S}}(T)$
 be the set of $q$-forms of the form 
$$\psi(z;\Xi, d\Xi) = \sum_{|m|=q} \psi_m(\Xi,z) \lambda_m(\Xi;d\Xi)$$
 which is closed as a form on $\Xi(V)$
 such that $\psi_m(\Xi, \cdot) \in \pi^*_{\Xi} \mathcal{S}(t(\Xi))$ for all $m$ and $\Xi \in \Xi(V)$.
Since John system for $F$ in (\ref{eqn:def_of_kappa}) is equivalent to the closedness of $\kappa F$ (see \cite[Page 82]{G98}),
 we obtain the following.}

{
\begin{theorem} \label{thm:image_kappa}
The image $\kappa \mathcal{S}(T)$ lies in $\tilde{\mathcal{S}}(T)$,
 and one has the following commutative diagram:
$$
\xymatrix{
\mathcal{S}(V) \ar[r]^{\mathcal{L}} \ar[dr]_{\tilde{\mathcal{L}}} & \mathcal{S}(T) \ar[d]_{\kappa} \\
& \tilde{\mathcal{S}}(T).} 
$$
\end{theorem}
}
{
\begin{remark}
It is possible to prove that $\kappa$ is isomorphism.
\end{remark} 
}

\subsection{Connection between cohomological Laplace transform and Fourier transform}
Let $\psi(z;\Xi, d\Xi) = \sum_m \psi_m(\Xi,z)\lambda_m(\Xi;d\Xi)$ be a $q$-form belonging to $\tilde{\mathcal{S}}(T)$.
Since $x \in \R^n_x$ belongs to the boundary of $T(\Xi)$ for any $\Xi \in \Xi(V)$,
 and since {$\psi_m(\Xi,z) \in \pi^*_{\Xi}\mathcal{S}(t(\Xi))$}, 
 the boundary value $\psi(x;\Xi, d\Xi)$ is well-defined
 by Theorem~\ref{thm:Paley-Wiener}.
{Now we assume a condition on $V$ that the homology group of generating cycles is one dimensional.}
Let $\gamma$ be a generating cycle in $\Xi(V)$.
By definition, 
 there exists a non-zero integer $c(\gamma)$ such that for any $n$-form $\eta(\xi, d\xi)$ on $V$,
\begin{equation} \label{eqn:pullback}
 \int_{\gamma}\int_{ \{\tau \,;\,\tau \Xi\, \in v(\Xi)\}} \Phi^* \eta 
 = c(\gamma) \int_V \eta 
\end{equation}
 holds.
We define \textit{the boundary value map} $b$ from $\tilde{\mathcal{S}}(T)$ into $C^{\infty}(\R^n_x)$ by
$$
 b\psi(x) := \frac{1}{c(\gamma)} \int_{\gamma} \psi(x;\Xi, d\Xi) \qquad (x \in \R^n_x).
$$


\begin{theorem} \label{thm:formL-Fourier}
One has $\mathcal{F} f = b \circ \tilde{\mathcal{L}}f$ for $f \in \mathcal{S}(V)$.
\end{theorem}

\begin{proof}
Substituting $\eta(\xi, d\xi) = e^{i \xi x} f(\xi)\,d\xi_1 \wedge \dots \wedge d\xi_n$ to (\ref{eqn:pullback}),
 we obtain by (\ref{eqn:def_of_formLf})
 for $x \in \R^n_x$ 
 \begin{align*}
 \int_{\gamma} \tilde{\mathcal{L}}f(x;\Xi,d\Xi)
 &=  \int_{\gamma}\int_{ \{\tau \,;\,\tau \Xi\, \in v(\Xi)\}} \Phi^*(e^{i \xi x} f(\xi)\,d\xi_1 \wedge \dots \wedge d\xi_n)\\ 
 &= c(\gamma) \int_{V} e^{i \xi x} f(\xi)\,d\xi_1 \wedge \dots \wedge d\xi_n\\
 &= c(\gamma) \mathcal{F} f(x),
 \end{align*}
which implies the statement. 
\end{proof}


\begin{remark}
At the image of cohomological Laplace transform,
 we have not all cohomology but just classes which are characterized by {special} closed forms.
We do not give here an invariant description.
Let us only mention that their decomposition on $\lambda_m$ means that,
 in the projection on the Grassmannian, 
 they are orthogonal to all Schubert cells but Euler's one.
The property to be constant on affine spaces $\set{z}{\Xi z = p}$ means that
 integrals of cohomology along these planes are finite.
\end{remark}

\subsection{Radon-John transform and Cauchy-Szeg\"o projection}
We have constructed the chain of transforms
$$
\xymatrix{
 \mathcal{S}(V) \ar[r]^{\mathcal{L}} \ar@/_{8pt}/[rr]_{\tilde{\mathcal{L}}}& \mathcal{S}(T) \ar[r]^{\kappa} & \tilde{\mathcal{S}}(T) \ar[r]^b & \mathcal{S}(\R^n_x)}, \quad
 f \mapsto \mathcal{L}f \mapsto \tilde{\mathcal{L}}f \mapsto \mathcal{F} f. 
$$
Now we want inverse this sequence using the Radon-John transform.
For $F(x) \in \mathcal{S}(\R^n_x)$,
 the Radon-John transform $\mathcal{J} F(\Xi,r)$ is defined as a function on $\mathrm{St}(\ell, n) \times \R^{\ell}_r$ given by
 the integration of $F(x)$ along the affine space $\set{x \in \R^n_x}{\Xi x = r}$.
The precise definition is as follows.
For $\Xi \in \mathrm{St}(\ell, n)$,
 let $\xi^j dx$ denote the linear form $\sum_{\alpha=1}^n \xi^j_{\alpha} dx_{\alpha}$,
 which is a matrix product of the row vector $\xi^j$ and
 the column vector $dx$.
Then $\bigwedge_{1 \le j \le l} \xi^j dx$ is an $\ell$-form on $\R^n_x$.
Let us define
\begin{equation} \label{eqn:def_of_RJ}
\mathcal{J} F(\Xi, r) := \int_{ \{x\,; \,\Xi x = r\}} F(x) \, \bigl( \bigwedge_{1 \le j \le l} \xi^j dx\bigr)\rfloor dx_1 \wedge \dots \wedge dx_n
\end{equation}
(\cite[Lecture 2, (1)']{G98} and \cite[(18)]{G13}),
 where    
 $(\bigwedge_{1 \le j \le l} \xi^j dx)\rfloor dx_1 \wedge \dots \wedge dx_n$　
 stands for any $q$-form $\eta$ on $\R^n_x$ for which
$$
\bigl( \bigwedge_{1 \le j \le l} \xi^j dx \bigr) \wedge \eta =  dx_1 \wedge \dots \wedge dx_n.
$$
Such $\eta$ is not unique,
 while the restriction of $\eta$ to the affine space $\set{x}{\Xi x = r}$ always coincides.
One can show the following.


\begin{lemma} \label{lemma:RJ-Fourier}
Let $f \in \mathcal{S}(\R^n_{\xi})$.
Then, for any $\Xi \in \mathrm{St}(\ell, n)$, the restriction of $f$ to $\R^{\ell}_{\tau}(\Xi)$
 and the Radon-John transform of $F = \mathcal{F} f$ are connected by ``small'' Fourier transform $\tau \to r$,
 that is,  
\begin{equation} \label{eqn:RJ-Fourier}
 \mathcal{J} F(\Xi,r) = (2\pi)^{n-\ell} \int_{\R^{\ell}_{\tau}} e^{i \tau r} f(\tau \Xi)\,d\tau_1\wedge  \dots \wedge d\tau_{\ell}.
\end{equation}
\end{lemma}


If our intersection of the cone with $\R^{\ell}_{\tau}(\Xi)$ would coincide with the slice,
then the holomorphic extension of Radon-John transform of $\mathcal{F} f$  would coincide with 
 $\mathcal{L}f$.
It explains why John system describes the image of the cohomological Laplace transform
(cf. the localization of the slice at the proof of Theorem~\ref{thm:Cech-cocycle}).

Since the slice is only one component of the intersection in general,
 we need to apply to the Radon-John transform
 the convolution dual to the multiplication on the characteristic function of the slice --
 Cauchy-Szeg\"o projection.

Let $F \in \mathcal{S}(\R^n_x)$.  
For each $\Xi \in \Xi(V)$,
 we regard $\mathcal{J} F(\Xi, r)$ as a function of $r \in \R^{\ell}_r(\Xi)$,
 and consider the Cauchy-Szeg\"o projection of $\mathcal{J} F(\Xi, r)$
 with respect to the tube $t(\Xi) = \R^{\ell}_r(\Xi) + i v^*(\Xi)$
 as in Section~1:
$$
\mathcal{K}_{v(\Xi)} (\mathcal{J} F)(\Xi,p) := K_{v(\Xi)} * (\mathcal{J} F)(\Xi, p) \qquad (p \in t(\Xi)).
$$
We define the Cauchy-Szeg\"o projection $\mathcal{K}_T F$ with respect to the non-convex tube $T$ by
\begin{equation} \label{eqn:def_of_CST}
 \mathcal{K}_T F(\Xi,z)
 := (2\pi)^{-q} \mathcal{K}_{v(\Xi)} (\mathcal{J} F)(\Xi, p)|_{p = \Xi z} \qquad (\Xi \in \Xi(V),\,\,z \in T(\Xi)).
\end{equation}


\begin{theorem} \label{thm:Szegoe}
{\rm (i)} For $f \in \mathcal{S}(V)$, one has $\mathcal{L}f = \mathcal{K}_T \circ \mathcal{F} f$.\\
{\rm (ii)} The spaces $\mathcal{S}(V)$ and $\mathcal{S}(T)$ are isomorphic.
\end{theorem}

\begin{proof}
(i)
Let $F = \mathcal{F} f$.
We see from (\ref{eqn:RJ-Fourier}) that
\begin{align*}
 \mathcal{K}_{v(\Xi)} (\mathcal{J} F)(\Xi,p)
 &= (2\pi)^{q}  \int_{\R^{\ell}_{\tau}} e^{i \tau p} \chi_{v(\Xi)}(\tau \Xi)f(\tau \Xi) \,d\tau_1\wedge  \dots \wedge d\tau_{\ell}\\
 &= (2\pi)^{q}  \int_{ \{\tau\,;\, \tau \Xi \, \in v(\Xi)\} } e^{i \tau p} f(\tau \Xi)\,d\tau_1\wedge  \dots \wedge d\tau_{\ell}.
\end{align*}
Therefore, in view of (\ref{eqn:def_of_CST})
 we obtain $\mathcal{K}_T F(\Xi,z) = \mathcal{L}f(\Xi, z)$.\\
(ii) For any $\psi(\Xi, z)\in \mathcal{S}(T)$,
 we see that $F=b(\kappa\psi) \in \mathcal{S}(\R^n_x)$ is well-defined. 
Let $f=\mathcal{F}^{-1}F$. 
We need to prove that $\mathcal{L}f=\psi$. 
Let us fix a generating cycle $\gamma$ and compare the representations of $F$ as Fourier transform and through $\psi$ for this cycle.
Then one can be obtained from another by the change of variables $\xi \to (\Xi,\tau)$ and the integrands coincide almost everywhere. 
Therefore we have $\psi=\mathcal{L}f$ on $\gamma$ and we can do it for any cycle.
\end{proof}

Now Theorem~\ref{thm:Image_of_cohL} follows from Theorem~\ref{thm:Szegoe}.
Let us draw a diagram for all spaces and operators introduced in this section:
\begin{equation} \label{eqn:big_diagram}
\xymatrix{ & & \mathcal{S}(T) \ar[ld]^{\kappa} & \\
\mathcal{S}(V) \ar[rru]^{\mathcal{L}} \ar[r]^{\tilde{\mathcal{L}}} \ar[rrd]_{\mathcal{F}} &
\tilde{\mathcal{S}}(T) \ar[rd]^b & & \mbox{(Image of Radon-John transform)} \ar[ul]_{\mathcal{K}_{v(\Xi)}} \\
& & \mathcal{S}(\R^n_x) \ar[uu]^{\mathcal{K}_T} \ar[ru]_{\mathcal{J}} & }
\end{equation}

\begin{remark}
Using the similar construction, we can define the cohomological space $\mathcal{S}'(T)$
 and the operator of boundary values $b: \mathcal{S}'(T) \to \mathcal{S}'(\R^n_x)$. 
However we cannot add it by a cohomological Laplace transform
 since for this we need to restrict distributions on slices or define Radon-John transform on $\mathcal{S}'$, which requests special justifications.
\end{remark}


\section{Cohomological Hardy space at non-convex tube $T$}
In Section 4,
 we have defined the isomorphic spaces of cohomology, $\mathcal{S}(T)$ and $\tilde{\mathcal{S}}(T)$,
 which are isomorphic images of $\mathcal{S}(V)$ by the cohomological Laplace transforms.
In this section,
 we want to define 
 cohomological analogue of Hardy norm on $\mathcal{S}(T)$, 
 and give $L^2$-version of the cohomological Laplace transform.
 
Similarly to $\pi_{\Xi}^* \mathcal{S}(t(\Xi))$ for $\Xi \in \Xi(V)$,
 let $\pi_{\Xi}^* \mathrm{H}(t(\Xi))$ denote
 the function space $\set{\phi \circ \pi_{\Xi}}{\phi \in \mathrm{H}(t(\Xi))}$ over $T(\Xi)$.
The inner product on $\pi_{\Xi}^* \mathrm{H}(t(\Xi))$ is defined by 
\begin{align*}
{}& \bigl( \psi_1,\, \psi_2 \bigr)_{\pi_{\Xi}^* \mathrm{H}(t(\Xi))} := \bigl( \phi_1,\, \phi_2 \bigr)_{\mathrm{H}(t(\Xi))}\\
 &\qquad (\psi_j := \phi_j \circ \pi_{\Xi},\,\, \phi_j \in \mathrm{H}(t(\Xi)),\,\,j=1,2).
\end{align*}
Let $F(\Xi,z)$ be a function belonging to $\mathcal{S}(T)$.
Noting that 
 $F(\Xi,\,\cdot\,) \in \pi_{\Xi}^* \mathcal{S}(t(\Xi)) \subset \pi_{\Xi}^*\mathrm{H}(t(\Xi))$,
 we define a $q$-form $|F|^2(\Xi, d\Xi)$ on the parameter set $\Xi(V)$ by
\begin{equation} \label{eqn:def_of_normform}
 |F|^2(\Xi, d\Xi) 
 := \bigl( F(\Xi,\,\cdot)\, , \, \kappa F(\cdot\,;\Xi, d\Xi) \bigr)_{\pi^*_{\Xi}\mathrm{H}(t(\Xi))}\\
\qquad (\Xi \in \Xi(V)),
\end{equation}
which we call \textit{the prenorm-form of} $F$.


\begin{lemma} \label{lemma:Hardynorm}
For $F \in \mathcal{S}(T)$,
 the prenorm-form $|F|^2(\Xi,d\Xi)$ is a closed form on $\Xi(V)$, 
 and the quantity
$\Vert F \Vert^2 := \frac{1}{c(\gamma)} \int_{\gamma} |F|^2(\Xi,d\Xi)$
 is independent of a choice of a generating cycle $\gamma$.
\end{lemma}

\begin{proof}
Thanks to Theorem~\ref{thm:Image_of_cohL},
 we can take $f \in \mathcal{S}(V)$ for which $F = \mathcal{F} f$. 
By the Plancherel formula,
 we have
{\begin{align*}
{}&\Bigl( \mathcal{L}_p f(\Xi,\cdot),\,
           \bigl(\frac{1}{i} \frac{\partial}{\partial p} \bigr)^m \mathcal{L}_p f(\Xi,\,\cdot) \Bigr)_{\mathrm{H}(t(\Xi))}\\
&= (2\pi)^{\ell} \int_{ \{\tau\,;\, \tau \Xi \, \in v(\Xi)\} } |f(\tau \Xi)|^2 \tau^m \, d\tau_1 \wedge \dots \wedge d\tau_{\ell}. 
\end{align*}
}Thus, by the argument similar to (\ref{eqn:Lf-det}),
 we get
\begin{align*}
 {}& |F|^2(\Xi,d\Xi)\\
 &= \sum_{|m|=q} 
   (2\pi)^{\ell} \Bigl( \int_{ \{\tau\,;\, \tau \Xi \, \in v(\Xi)\} } |f(\tau \Xi)|^2 \tau^m \, d\tau_1 \wedge \dots \wedge d\tau_{\ell} \Bigr)
   \lambda_m(\Xi; d\Xi)\\
 &= \frac{(2\pi)^{\ell} }{\ell! (n-\ell)!} 
     \int_{ \{\tau\,;\, \tau \Xi \, \in v(\Xi)\} } |f(\tau \Xi)|^2\,d\tau_1 \wedge \dots \wedge d\tau_{\ell}
     \wedge \det\begin{bmatrix}\Xi \\ (\sum_{j=1}^{\ell} \tau_j d\xi^j)^{ \{ q \} } \end{bmatrix} \\
 &=  (2\pi)^{\ell} \int_{ \{\tau\,;\, \tau \Xi \, \in v(\Xi)\} } \Phi^*( |f(\xi)|^2 d\xi_1 \wedge \dots \wedge d\xi_n).
\end{align*}
Therefore $|F|^2(\Xi,d\Xi)$
 is the direct image of the closed $n$-form $\Phi^*( |f(\xi)|^2 d\xi_1 \wedge \dots \wedge d\xi_n)$,
 so that $|F|^2(\Xi,d\Xi)$ is closed.
On the other hand, by (\ref{eqn:pullback}) we have
\begin{equation} \label{eqn:isometry}
\begin{aligned}
\Vert F \Vert^2 &= \frac{(2\pi)^{\ell}}{c(\gamma)}\int_{\gamma} \int_{ \{\tau\,;\, \tau \Xi \, \in v(\Xi)\} } \Phi^*( |f(\xi)|^2 d\xi_1 \wedge \dots \wedge d\xi_n)\\
 &= (2\pi)^{\ell} \int_V |f(\xi)|^2 \,d\xi_1 \wedge \dots \wedge d\xi_{n}.
\end{aligned}
\end{equation}
Therefore $\Vert F \Vert^2$ is independent of the choice of a generating cycle $\gamma$.
\end{proof}

 
Let us call $\Vert F \Vert$ \textit{the cohomological Hardy norm} on $\mathcal{S}(T)$.
Thanks to (\ref{eqn:isometry}), we obtain the following.


\begin{theorem} \label{thm:isometry}
One has an isometric isomorphism
 $$\mathcal{S}(V) \owns f \mapsto (2\pi)^{-\ell/2} \mathcal{L}f \in \mathcal{S}(T)$$
 with respect to the $L^2$-norm
 and the cohomological Hardy norm.
\end{theorem}

 
Let $\mathrm{H}(T)$ be the completion of $\mathcal{S}(T)$ with respect to the cohomological Hardy norm.
We call $\mathrm{H}(T)$ \textit{the cohomological Hardy space} over $T$.
Theorem~\ref{thm:isometry} together with the diagram (\ref{eqn:big_diagram}) tells us that
 the map $b \circ \kappa : \mathcal{S}(T) \to \mathcal{S}(\R^n_x)$ and the Cauchy-Szeg\"o projection
 $\mathcal{K}_T : \mathcal{S}(\R^n_x) \to \mathcal{S}(T)$
 are isometries up to scalar multiple.
These operators are extended continuously to the boundary value map $\mathrm{H}(T) \to L^2(\R^n_x)$ and the $L^2$ Cauchy-Szeg\"o projection
 $L^2(\R^n_x) \to \mathrm{H}(T)$ respectively.
They are isometric up to scalar multiple again, and the $L^2$ Cauchy-Szeg\"o projection is a left inverse of the boundary value map.


\end{document}